\documentclass[11pt,a4paper]{article}

\usepackage{inputenc}
\usepackage{amsmath}
\usepackage{bm}
\usepackage{bbold}
\usepackage{amsthm}
\usepackage{enumerate}

\usepackage[top=1.0in, left=1.0in, right=1.0in, bottom=1.0in]{geometry}

\usepackage{tikz}\tikzset{x=1cm,y=1cm,z=1cm}

\usepackage{pgfplots}\pgfplotsset{compat=1.16}

\usepackage[hyphens]{url}

\usepackage{hyperref}
\usepackage{breakurl}

\title{Tropical modeling of battery swapping and charging station\thanks{Mathematics 12 (5), 644 (2024)}}

\author{N. Krivulin\thanks{Faculty of Mathematics and Mechanics, St.~Petersburg State University, 28 Universitetsky Ave., St.~Petersburg, 198504, Russia; 
nkk@math.spbu.ru.}
\and
A. Garg\thanks{Department of Industrial and Manufacturing Systems Engineering, School of Mechanical Science and Engineering, Huazhong University of Science and Technology, Wuhan, 430074, PR China; akhilgarg@hust.edu.cn}
}

\date{}

\newtheorem{theorem}{theorem}
\newtheorem{lemma}[theorem]{lemma}

\newtheorem{proposition}[theorem]{proposition}

\theoremstyle{definition}

\begin{document}

\maketitle

\begin{abstract}
We propose and investigate a queueing model of a battery swapping and charging station (BSCS) for electric vehicles (EVs). A new approach to the analysis of the queueing model is developed, which combines the representation of the model as a stochastic dynamic system with the use of the methods and results of tropical algebra, which deals with the theory and applications of algebraic systems with idempotent operations. We describe the dynamics of the queueing model by a system of recurrence equations that involve random variables (RVs) to represent the interarrival time of incoming EVs. A performance measure for the model is defined as the mean operation cycle time of the station. Furthermore, the system of equations is represented in terms of the tropical algebra in vector form as an implicit linear state dynamic equation. The performance measure takes on the meaning of the mean growth rate of the state vector (the Lyapunov exponent) of the dynamic system. By applying a solution technique of vector equations in tropical algebra, the implicit equation is transformed into an explicit one with a state transition matrix with random entries. The evaluation of the Lyapunov exponent reduces to finding the limit of the expected value of norms of tropical matrix products. This limit is then obtained using results from the tropical spectral theory of deterministic and random matrices. With this approach, we derive a new exact formula for the mean cycle time of the BSCS, which is given in terms of the expected value of the RVs involved. We present the results of the Monte Carlo simulation of the BSCS's operation, which show a good agreement with the exact solution. The application of the obtained solution to evaluate the performance of one BSCS and to find the optimal distribution of battery packs between stations in a network of BSCSs is discussed. The solution may be of interest in the case when the details of the underlying probability distributions are difficult to determine and, thus, serves to complement and supplement other modeling techniques with the need to fix a distribution.
\\

\\

\end{abstract}

\section{Introduction}

The usage of electric vehicles (EVs) has seen a great rise on a large scale in recent years \cite{Chau2014Pure,Lebrouhi2021Key}. However, the adoption of electric vehicles is limited by problems such as the slow charging of battery packs (BPs) and the accelerated aging of BPs during fast charging (see, e.g., \cite{Hemavathi2022Study,Cui2023Analysis} for overviews of related problems and solution trends). Since batteries are the main source of power for EVs, ensuring energy supply is an important way to improve users' experience. At present, the power supply method for EVs is divided into two types: plug-in charging and battery swaps \cite{Ji2018Plugin,Hemavathi2022Study}. Plug-in charging has disadvantages such as a long charging time, fast charging shortening the service life of the battery, and the parking lots required for charging taking up a larger space. In addition, if the daily load of residents and the peak of EVs' charging load are in the same time period, this will lead to a ``peak--add--peak'' state, which will affect the normal operation of the power grid. On the contrary, the battery swapping scenario addresses these problems well. Battery swapping for EVs can decrease user waiting time, reduce purchase cost, and improve batteries' useful life. Therefore, many companies have adopted the battery swapping scenario for EVs. However, there are still some challenges in the promotion of the battery swapping scenario for EVs, such as the operating cost of the battery swapping and charging station (BSCS) and the centralized battery charging load.

Modern research on the implementation of battery swapping offers a range of models to study various aspects of the BSCS's operations, including the battery logistics and transportation strategy as well as energy management and operation scheduling in EVs' battery swapping and charging systems and networks. A queueing model of BSCSs with Poisson arrival and constant service times was proposed by Choi and Lim in \cite{Choi2020Analysis} to develop and analyze queue-length-dependent overload control policies. The queue distributions under different policies are derived using an embedded Markov chain, and system performance measures such as blocking probability and mean waiting time are examined by numerical examples.

In \cite{Asadi2022Monotone}, the operation of a BSCS was represented using a finite-horizon Markov decision process model combined with a dynamic programming algorithm, which allows determining the number of BPs to recharge, discharge, and replace over time. In \cite{You2020Scheduling}, an optimal scheduling problem was examined, which assigns the best BSCS to each EV, based on their current location and battery charge level. A scheduling strategy of the optimal transportation of BPs from a charging station to a swapping station was developed in \cite{Li2022Battery}. The strategy involves an optimization problem, which is solved using a genetic algorithm. This strategy was compared with two simple strategies by using the Monte Carlo simulation of battery swapping demand.

Given the cost and efficiency of operating battery replacement and charging stations, many scientists have proposed various strategies. Kang et al. \cite{Kang2016Centralized} proposed a new strategy for the centralized charging of EVs within the framework of battery replacement scenarios. This strategy considers optimal charging priorities and charging locations and, ultimately, minimizes the total charging cost based on electricity spot prices. A battery-planning strategy based on a partitioned battery management method was developed by \linebreak Yang et al. \cite{Yang2019Optimal}. In order to maximize profits, an optimization objective function was set, including the number of batteries in each segment. San et al. \cite{Sun2018Optimal} obtained the optimal charging strategy for a single battery replacement and charging station in order to minimize the cost of charging it. The optimization model was transformed into a Markov decision process with constraints, and the optimal strategy was obtained using the Lagrangian method and dynamic programming. Liu et al. \cite{Liu2019Distributed} demonstrated the method to optimize the charging and logistics of discharged and fully charged batteries to maximize the profits of battery replacement and charging stations. The problem of the optimal location of battery swapping stations (BSSs) was solved in \cite{Maghfiroh2023Location} by using a fuzzy multi-criteria decision-making~approach.

Wang et al. \cite{Wang2019Qos} addressed the problem of the online management of BSSs to minimize energy costs and ensure service quality. At the same time, the problem of designing optimal autonomous battery-swapping stations was studied to determine the optimal number of batteries and to achieve the ideal compromise between charging flexibility and battery cost distribution. Zhang et al. \cite{Zhang2019Demand} confirmed that the peak demand for battery replacement services can be effectively reduced, and the use of optimization strategies can reduce the total cost by approximately 12\%. The above research mainly considered the impact of factors such as the number of batteries, the battery charging costs, logistics planning, electricity prices during use, and the profits of battery replacement stations on the cost of battery replacement and charging stations to reduce operating costs. It suggested using different strategies to replace batteries and charging stations to varying degrees. However, the above view does not take into account the relationship between the load generated by large-scale centralized charging and the cost of replacing batteries and operating a charging station.

Huang et al. \cite{Huang2017Multi} found that the correspondence between the stochastic supply of wind and the need to charge electric vehicles can reduce the demand for conventional energy and carbon dioxide emissions. Xing et al. \cite{Xing2019Charging} proposed a data-based method for predicting electric vehicle charging demand based on online travel data. The predictive model provides recommendations for chargers and charging management. Battery charging load research mainly includes load prediction and the impact of reducing battery charging load on the power grid. The impact of reducing battery charging load is mainly focused on charging methods and increasing the overload capacity of the power grid. 

However, there are few studies on the logistics and transportation of batteries at swapping and charging stations. Furthermore, current research does not propose specific transportation strategies to reduce the load generated by large-scale charging. The operating costs of battery replacement and charging stations and large-scale battery charging loads are the main areas for the implementation of battery replacement scenarios \cite{Wang2014Centralized,Yang2019Decentralized}.

With the aim to contribute to the study of the BSCS's operations in this paper, the authors propose a framework based on tropical algebra to represent and analyze stochastic models of BSCSs. Tropical (idempotent) algebra deals with the theory and applications of algebraic systems with idempotent operations \cite{Baccelli1993Synchronization,Kolokoltsov1997Idempotent,Golan2003Semirings,Heidergott2006Maxplus,Gondran2008Graphs,Krivulin2009Methods,Butkovic2010Maxlinear}. A typical example of these systems is max-plus algebra, which is a semifield with addition defined as the operation of the maximum and multiplication as the arithmetic addition. One of the advantages of tropical algebra is that many problems, which are not linear in the ordinary sense, can turn into linear ones in the tropical algebra setting. The models and methods of tropical algebra find applications in various research domains such as location analysis, project scheduling, and decision-making. The application area includes stochastic dynamic systems, where tropical algebra serves as a useful tool to represent and analyze stochastic systems governed by tropical linear dynamic equations \cite{Heidergott2006Maxplus,Heidergott2006Maxplus-linear,Krivulin2009Methods}.

In this paper, we propose and investigate a queueing model of BSCSs for EVs. A new approach to the analysis of the queueing model is developed, which combines the representation of the model as a stochastic dynamic system with the application of the methods and results of tropical algebra. We describe the dynamics of the queueing model by a system of recurrence equations that involve random variables (RVs) to represent the interarrival time of incoming EVs. A performance measure for the model is defined as the mean operation cycle time of the station. Furthermore, the system of equations is represented in terms of the max-plus algebra in vector form as an implicit linear state dynamic equation. The performance measure takes on the meaning of the mean growth rate of the state vector (the Lyapunov exponent) of the dynamic system. By applying a solution technique of vector equations in tropical algebra, the implicit equation is transformed into an explicit one with a state transition matrix with random entries. The evaluation of the Lyapunov exponent reduces to finding the limit of the expected value of norms of tropical matrix products. This limit is then obtained using results from the tropical spectral theory of deterministic and random matrices. With this approach, we derive a new exact formula for the mean cycle time of the BSCS, which is given in terms of the expected value of the RVs involved. The numerical results of the Monte Carlo simulation of the BSCS's operation is demonstrated, which show a good agreement with the exact solution. We discuss the applications of the obtained solution to evaluate the performance of one BSCS and to find the optimal distribution of BPs between stations in a network of BSCSs.

As the obtained results show, the proposed approach based on the methods and results of tropical algebra presents a useful tool to model and analyze some classes of queueing models including the model of the BSCS under study. The approach can offer solutions, which are given in terms of the expected values of the RVs involved and independent of the particular probability distribution of the RVs. Such solutions may be of interest in the case when the details of the underlying probability distributions are difficult to determine and, thus, serve to supplement other modeling techniques with the need to fix a distribution.  

The rest of the paper is organized as follows. In Section~\ref{S-BSCSM}, we describe a queueing model of the BSCS that serves to both motivate and illustrate the study. Section~\ref{S-ETA} provides an overview of key definitions and the notation, and presents the preliminary results of tropical algebra, which are used in subsequent sections to examine the model under consideration. A stochastic dynamic model defined in the tropical algebra setting is described, and some related results are discussed in Section~\ref{S-SDS}. An application of the tropical-algebra-based approach to the analysis of the BSCS model is demonstrated in Section~\ref{S-ABSCSM}, which includes a simple formula for calculating the mean cycle time of the BSCS and related simulation results. In Section~\ref{S-EAP}, an example of the application of the obtained results to the optimal distribution of BPs between BSCSs is illustrated. Section~\ref{S-S} offers some concluding remarks.

\section{Battery Swapping and Charging Station Model}
\label{S-BSCSM}

We considered a BSCS that serves incoming requests of EVs to swap a depleted (discharged) BP to a fully charged one. Each EV is assumed equipped with one BP, and all BPs are considered of the same type (identical). The BSCS consists of a battery swapping and battery charging/storage areas.

The station has a set of identical BPs located in the storage area where they are charging and, then, waiting for use in swapping. All BPs can be charged simultaneously, and the charging of every BPs takes the same time. The swapping operations are performed one at a time and require equal time for all EVs.

The EVs arrive at the BSCS at random with time intervals distributed according to some probability law. Upon arrival, an EV waits until the following conditions hold: (i) a fully charged BP is available, and (ii) the swapping of the BP for the previous EV is completed, or the swapping procedure is immediately started if a fully charged BP and the swapping unit are both available.

A graphical representation of a BSCS as a queueing model is given in Figure~\ref{F-QMBSCS}. The model consists of (i) a single-server queue, which represents an arrival source of EVs, (ii) a single-server fork--join queue, which represents the swapping of batteries, and (iii) a multi-server queue, which represents the charging of the batteries. All queues have infinite buffers. At the initial time, the first queue is assumed to have an infinite number of jobs (EVs), the second queue has no job (the EV and BP ready for swapping), and the third queue has $m$ jobs (BPs ready for charging). 
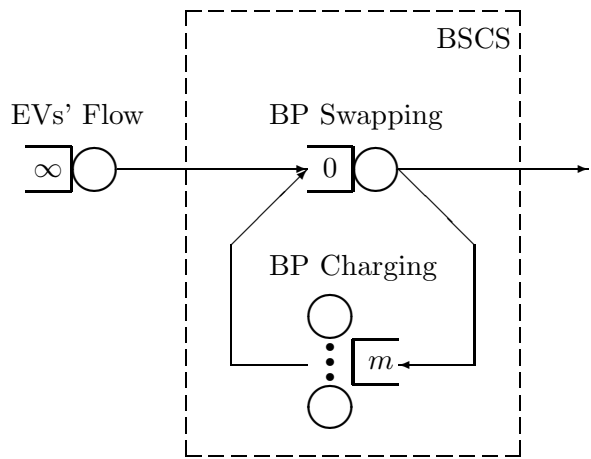
\begin{figure}[ht]
\setlength{\unitlength}{1.0mm}
\begin{center}
\begin{picture}(75,60)

\newsavebox\queue
\savebox{\queue}(10,6){\thicklines
 \put(0,3){\line(1,0){6}}
 \put(0,-3){\line(1,0){6}}
 \put(6,3){\line(0,-1){6}}
 \put(9,0){\circle{6}}}

\newsavebox\lqueue
\savebox{\lqueue}(10,6){\thicklines
 \put(3,0){\circle{6}}
 \put(6,3){\line(1,0){6}}
 \put(6,-3){\line(1,0){6}}
 \put(6,3){\line(0,-1){6}}}

\newsavebox\lqueues
\savebox{\lqueues}(10,18){\thicklines
 \put(3,-6){\circle{6}}

 \multiput(3,-2)(0,2){3}{\circle*{1}}

 \put(3,6){\circle{6}}

 \put(6,3){\line(1,0){6}}
 \put(6,-3){\line(1,0){6}}
 \put(6,3){\line(0,-1){6}}}

\newsavebox\queuef
\savebox{\queuef}(10,6){\thicklines
 \put(0,3){\line(1,0){6}}
 \put(0,-3){\line(1,0){6}}
 \put(0,3){\line(0,-1){6}}
 \put(6,3){\line(0,-1){6}}
 \put(9,0){\circle{6}}}

\newsavebox\lqueuef
\savebox{\lqueuef}(10,6){\thicklines
 \put(3,0){\circle{6}}
 \put(6,3){\line(1,0){6}}
 \put(6,-3){\line(1,0){6}}
 \put(6,3){\line(0,-1){6}}
 \put(12,3){\line(0,-1){6}}}

\put(3,35){\usebox\queue}
\put(15,38){\vector(1,0){25}}
\put(4,37){$\infty$}
\put(1,44){EVs' Flow}

\put(40,35){\usebox\queue}
\put(52,38){\vector(1,0){25}}
\put(42,37){$0$}
\put(35,44){BP Swapping}

\put(52,38){\line(1,-1){10}}
\put(62,28){\line(0,-1){16}}
\put(62,12){\vector(-1,0){10}}
\put(40,2){\usebox\lqueues}
\put(40,12){\line(-1,0){10}}
\put(30,12){\line(0,1){16}}
\put(30,28){\vector(1,1){10}}
\put(48,11.5){$m$}
\put(35,24){BP Charging}

\multiput(24,0)(3,0){15}{\line(1,0){2}}
\multiput(24,0)(0,3){20}{\line(0,1){2}}

\multiput(24,59)(3,0){15}{\line(1,0){2}}
\multiput(68,0)(0,3){20}{\line(0,1){2}}

\put(57,54){BSCS}

\end{picture}
\end{center}
\caption{Queueing model of battery swapping and charging station.}\label{F-QMBSCS}
\end{figure}

\subsection{System of Recurrence Equations}

Suppose that the station has $m$ BPs intended for swapping, and define the following state variables. For $k=1,2,\ldots$, let $x(k)$ be the arrival epoch of the $k$th incoming EV, $y(k)$ be the completion time of the battery swapping for the $k$th EV, and $z(k)$ be the time when a fully charged BP is available for the $k$th EV. We also assumed that $x(k)=y(k)=z(k)=0$ for all $k\leq0$.

We now describe the evolution of the system as a set of recurrence equations. We denote the time interval between the $(k-1)$the and $k$th arrival epochs by $\alpha_{k}$ and assume $\{\alpha_{k}|\ k=1,2,\ldots\}$ to be a sequence of independent and identically distributed positive (nonnegative) RVs with a finite expected value $\mathsf{E}\alpha_{1}=a\geq0$ and variance $\mathsf{D}\alpha_{1}$. With this notation, we can represent the $k$th arrival epoch as
\begin{equation*}
x(k)
=
x(k-1)+\alpha_{k}.
\end{equation*}

Furthermore, we denote by $b>0$ the swapping time of one BP. Observing that the $k$th swapping operation starts as soon as the following events occur: (i) the $k$th EV arrives, (ii)~the $(k-1)$the swapping operation completes, and (iii) a fully charged BP becomes available for the $k$th time, we write the equation:
\begin{equation*}
y(k)
=
\max(x(k),y(k-1),z(k))+b.
\end{equation*}

Finally, we assumed that all $m$ BPs available in the station are discharged at the initial time epoch $k=0$. With the charging time of one BP denoted by $c>0$, we have
\begin{equation*}
z(k)
=
y(k-m)+c.
\end{equation*}

We now substitute $z(k)$ from the last equation into the second and, then, combine the first and second equations into the dynamic system in two state variables:
\begin{equation}
\begin{aligned}
x(k)
&=
x(k-1)+\alpha_{k},
\\
y(k)
&=
\max(x(k),y(k-1),y(k-m)+c)+b.
\end{aligned}
\label{E-xkeqxk1alphak-ykeqmaxxkyk1cb}
\end{equation}

\subsection{Performance Measure}

We define the operation cycle of the BSCS model described by \eqref{E-xkeqxk1alphak-ykeqmaxxkyk1cb} as the interval between successive completions of swapping operations. Furthermore, we consider the mean (average) cycle time over the first $k$ cycles:
\begin{equation*}
\frac{1}{k}\sum_{i=1}^{k}(y(i)-y(i-1))
=
\frac{1}{k}y(k).
\end{equation*}

We turn to the limit when $k$ tends to $\infty$ and assume that this limit exists (deterministically or with probability one) to write
\begin{equation}
\lim_{k\to\infty}
\frac{1}{k}y(k)
=
\lambda.
\label{E-lambda}
\end{equation}

The constant $\lambda$ is referred to as the mean cycle time and may serve as a useful characteristic of the system. Specifically, for a large time horizon $T$, the ratio $T/\lambda$ differs little from the mean number of battery swaps in the time interval from $0$ to $T$. Since this ratio also shows the mean number of batteries swapped, it can be used to estimate other characteristics such as the mean total energy consumption for battery charging (which is considered proportional to the mean total charging time estimated as $cT/\lambda$, where $c$ denotes the energy consumption per BP) or the mean total revenue received from customers (proportional to $T/\lambda$).

The evaluation of the mean cycle time directly from recurrent Equation \eqref{E-xkeqxk1alphak-ykeqmaxxkyk1cb}, which involves the operation of the maximum, may be a rather difficult problem even if the equations are in a simple form like those presented above. At the same time, this form of the recurrence equations offers a potential for the use of models and methods of tropical algebra, which allows one to handle such equations in a unified analytical framework. In the subsequent sections, we show how to represent the equations in terms of tropical algebra in vector form and, then, use this representation to evaluate the mean cycle time analytically.

\section{Elements of Tropical Algebra}
\label{S-ETA}

We begin with preliminary definitions and results of tropical algebra, which are used for the representation and analysis of the dynamic model in what follows. Tropical (idempotent) algebra deals with the theory and applications of algebraic systems with idempotent operations, which are studied in many works, including the monographs \cite{Baccelli1993Synchronization,Kolokoltsov1997Idempotent,Golan2003Semirings,Heidergott2006Maxplus,Gondran2008Graphs,Krivulin2009Methods,Butkovic2010Maxlinear}. 

\subsection{Idempotent Semifield}
Let $\mathbb{X}$ be a set that is closed under associative and commutative binary operations: addition $\oplus$ and multiplication $\otimes$, and includes their neutral elements: zero $\mathbb{0}$ and identity $\mathbb{1}$. Addition is idempotent: $x\oplus x=x$ for all $x\in\mathbb{X}$. Multiplication distributes over addition and is invertible: for each $x\ne\mathbb{0}$, there is an inverse $x^{-1}$ such that $x\otimes x^{-1}=\mathbb{1}$ (hereafter, the multiplication sign $\otimes$ is omitted to save writing). 

The power notation with integer exponents specifies iterated products: $x^{p}=x^{p-1}x$, $x^{-p}=(x^{p})^{-1}$, $\mathbb{0}^{p}=\mathbb{1}$, and $x^{0}=\mathbb{1}$ for $x\ne\mathbb{0}$ and integer $p>0$. The powers with rational exponents are also assumed well-defined. The binomial identity takes the form of the equality $(x\oplus y)^{q}=x^{q}\oplus y^{q}$, which is valid for any rational $q\geq0$.

The set $\mathbb{X}$ is assumed to be totally ordered by an order relation consistent with that induced by idempotent addition by the rule: $x\leq y$ if and only if $x\oplus y=y$. With respect to this order, addition and multiplication are monotone in each argument: if the inequality $x\leq y$ holds, then $x\oplus z\leq y\oplus z$ and $xz\leq yz$ for any $z$. For nonzero $x$ and $y$ such that $x\leq y$ and rational $q$, the inequality $x^{q}\leq y^{q}$  holds if $q\geq0$ and $x^{q}\geq y^{q}$ if $q<0$. Furthermore, the inequalities $x\leq x\oplus y$ and $y\leq x\oplus y$ are valid for all $x$ and $y$. Finally, the inequality $x\oplus y\leq z$ is equivalent to the system of inequalities $x\leq z$ and $y\leq z$.

The algebraic system $(\mathbb{X},\oplus,\otimes,\mathbb{0},\mathbb{1})$ is usually referred to as the idempotent semifield.

A typical example of the system is the real semifield $(\mathbb{R}\cup\{-\infty\},\max,+,-\infty,0)$, also known as max-plus algebra. In max-plus algebra, the operations are defined as $\oplus=\max$ and $\otimes=+$ and the neutral elements as $\mathbb{0}=-\infty$ and $\mathbb{1}=0$. For any $x\in\mathbb{R}$, the multiplicative inverse $x^{-1}$ is equal to the opposite number $-x$ in the standard arithmetics. The power $x^{y}$ coincides with the usual arithmetic product $x\times y$. The order relation $\leq$ corresponds to the natural linear order on $\mathbb{R}$.

\subsection{Algebra of Matrices and Vectors}
The scalar operations $\oplus$ and $\otimes$ are extended to vectors and matrices over $\mathbb{X}$ in the usual way. A matrix with all entries equal to $\mathbb{0}$ is the zero matrix denoted $\bm{0}$. For any matrices $\bm{A}=(a_{ij})$, $\bm{B}=(b_{ij})$ and $\bm{C}=(c_{ij})$ of appropriate sizes, and scalar $x$, matrix addition, matrix multiplication, and scalar multiplication are defined by componentwise formulas:
\begin{equation*}
(\bm{A}\oplus\bm{B})_{ij}
=
a_{ij}\oplus b_{ij},
\qquad
(\bm{A}\bm{C})_{ij}
=
\bigoplus_{k}a_{ik}c_{kj},
\qquad
(x\bm{A})_{ij}
=
xa_{ij}.
\end{equation*}

For any nonzero $(m\times n)$-matrix $\bm{A}=(a_{ij})$, its multiplicative conjugate is the $(n\times m)$-matrix $\bm{A}^{-}=(a_{ij}^{-})$ with the entries $a_{ij}^{-}=a_{ji}^{-1}$ if $a_{ji}\ne\mathbb{0}$, and $a_{ij}^{-}=\mathbb{0}$ otherwise.

The monotonicity of scalar addition and multiplication, as well as other properties that involve the order relations are extended to the operations on matrices, where the inequalities are understood componentwise. 

A square matrix is diagonal if all its off-diagonal entries are equal to $\mathbb{0}$ and triangular if its entries either above or below the diagonal are equal to $\mathbb{0}$. A triangular matrix with all diagonal entries equal to $\mathbb{0}$ is called strictly triangular. The block diagonal and (strictly) block triangular matrices are introduced in a similar way.

A diagonal matrix that has all diagonal entries equal to $\mathbb{1}$ is the identity matrix denoted by $\bm{I}$. The power notation for matrices is defined in the sense of tropical algebra as follows: $\bm{A}^{0}=\bm{I}$, $\bm{A}^{p}=\bm{A}\bm{A}^{p-1}$ for any square matrix $\bm{A}$ and integer $p>0$.

The trace of a square matrix $\bm{A}=(a_{ij})$ of order $n$ is given by
\begin{equation*}
\mathop\mathrm{tr}\bm{A}
=
a_{11}\oplus\cdots\oplus a_{nn}
=
\bigoplus_{i=1}^{n}a_{ii}.
\end{equation*}
 
A tropical analogue of the matrix determinant is defined as
\begin{equation*}
\mathop\mathrm{Tr}(\bm{A})
=
\mathop\mathrm{tr}\bm{A}\oplus\cdots\oplus\mathop\mathrm{tr}\bm{A}^{n}
=
\bigoplus_{m=1}^{n}\mathop\mathrm{tr}\bm{A}^{m}.
\end{equation*}

If the condition $\mathop\mathrm{Tr}(\bm{A})\leq\mathbb{1}$ holds, the Kleene star matrix is calculated as
\begin{equation*}
\bm{A}^{\ast}
=
\bm{I}\oplus\bm{A}\oplus\cdots\oplus\bm{A}^{n-1}
=
\bigoplus_{m=0}^{n-1}\bm{A}^{m}.
\end{equation*}

A matrix that consists of one column (row) is a column (row) vector. All vectors are assumed column vectors unless transposed. A vector with all entries equal to $\mathbb{0}$ is the zero vector denoted $\bm{0}$. Any vector that has no zero entries is called regular. The vector that has all entries equal to $\mathbb{1}$ is denoted by $\bm{1}=(\mathbb{1},\ldots,\mathbb{1})^{T}$. In max-plus algebra, where $\mathbb{1}=0$, the vector $\bm{1}$ has all entries equal to arithmetic zero 
 (the usual zero vector).

For any matrix $\bm{A}=(a_{ij})$ and vector $\bm{x}=(x_{i})$, tropical norms are given by
\begin{equation*} 
\|\bm{A}\|
=
\bm{1}^{T}\bm{A}\bm{1}
=
\bigoplus_{i,j}a_{ij},
\qquad
\|\bm{x}\|
=
\bm{1}^{T}\bm{x}
=
\bm{x}^{T}\bm{1}
=
\bigoplus_{i}x_{i},
\end{equation*} 
which coincide in max-plus algebra with the maximum entries of $\bm{A}$ and $\bm{x}$.

For any conforming matrices $\bm{A}$, $\bm{B}$ and $\bm{C}$, and scalar $x$, the following relations hold:
\begin{equation*}
\|\bm{A}\oplus\bm{B}\|
=
\|\bm{A}\|\oplus\|\bm{B}\|,
\qquad
\|\bm{A}\bm{C}\|
\leq
\|\bm{A}\|\|\bm{C}\|,
\qquad
\|x\bm{A}\|
=
x\|\bm{A}\|.
\end{equation*}

A scalar $\lambda$ is an eigenvalue of an $(n\times n)$-matrix $\bm{A}$ if there exists an $n$-vector $\bm{x}\ne\bm{0}$ such that $\bm{A}\bm{x}=\lambda\bm{x}$. The spectral radius of $\bm{A}$ is the maximum eigenvalue, which is given by
\begin{equation*}
\rho(\bm{A})
=
\mathop\mathrm{tr}\bm{A}\oplus\cdots\oplus\mathop\mathrm{tr}\nolimits^{1/n}(\bm{A}^{n})
=
\bigoplus_{m=1}^{n}\mathop\mathrm{tr}\nolimits^{1/m}(\bm{A}^{m}).
\end{equation*}

Note that, if the spectral radius is defined in the framework of max-plus algebra, it can be represented using ordinary arithmetic operations as the maximum of the mean (average) cyclic sums of entries in $\bm{A}$ in the form:
\begin{equation}
\rho(\bm{A})
=
\max\left\{a_{11},\ldots,a_{nn},\frac{a_{12}+a_{21}}{2},\ldots,\frac{a_{n-1,n}+a_{n,n-1}}{2},\ldots\right\}.
\label{E-rhoA}
\end{equation}

If a matrix $\bm{A}$ has no entries equal to $\mathbb{0}$, then for all integers $k\geq0$, the following inequality holds (see, e.g., \cite{Krivulin2007Convergence}):
\begin{equation}
\|\bm{A}^{k}\|
\leq
\rho^{k}(\bm{A})\|\bm{A}\bm{A}^{-}\|.
\label{I-AklerhoA}
\end{equation}

The next theorem is a consequence of the results obtained in \cite{Vorobjev1963Extremal,Romanovskii1967Optimization} (see also \cite{Kolokoltsov1997Idempotent,Krivulin2007Convergence}).
\begin{theorem}
\label{T-limAkl}
For any $(n\times n)$-matrix $\bm{A}$, there exist the limits:
\begin{equation*}
\lim_{k\to\infty}\|\bm{A}^{k}\|^{1/k}
=
\rho(\bm{A}),
\qquad
\lim_{k\to\infty}\mathop\mathrm{tr}\nolimits^{1/k}(\bm{A}^{k})
=
\rho(\bm{A}).
\end{equation*}
\end{theorem}

\subsection{Vector Equation and Matrix Inequality}
To conclude the overview of the preliminary results, we give a solution for a vector equation and derive inequalities for products of square matrices to be used in what follows.

Suppose that given an $(n\times n)$-matrix $\bm{A}$ and $n$-vector $\bm{b}$, the problem is to find regular $n$-vectors $\bm{x}$ that satisfy the equation:
\begin{equation}  
\bm{A}\bm{x}\oplus\bm{b}
=
\bm{x}.
\label{E-Axbeqx}
\end{equation}  

The following lemma offers a solution of the equation in a special case as a consequence of the results obtained in \cite{Krivulin2006Solution,Krivulin2009Methods}.
\begin{lemma}
\label{L-Axbeqx}
If $\mathop\mathrm{Tr}(\bm{A})<\mathbb{1}$, then Equation \eqref{E-Axbeqx} has the unique solution $\bm{x}=\bm{A}^{\ast}\bm{b}$.
\end{lemma}

We now turn to evaluating the lower and upper bounds for the norm of a product of matrices in block triangular form. Let $\bm{A}(i)$ for all $i=1,\ldots,k$ be conforming block triangular matrices given by the sum of block diagonal and strictly triangular matrices as follows:
\begin{equation}
\bm{A}(i)
=
\bm{D}(i)
\oplus
\bm{T}(i),
\qquad
\bm{D}(i)
=
\begin{pmatrix}
\bm{D}_{1}(i) & \bm{0}
\\
\bm{0} & \bm{D}_{2}(i)
\end{pmatrix},
\qquad
\bm{T}(i)
=
\begin{pmatrix}
\bm{0} & \bm{T}_{12}(i)
\\
\bm{0} & \bm{0}
\end{pmatrix}.
\label{E-AieqDiTi}
\end{equation}

Consider the product of the matrices $\bm{A}(i)$ over all $i=1,\ldots,k$, and denote it by
\begin{equation*}
\bm{A}_{k}
=
\bigotimes_{i=1}^{k}
\bm{A}(i)
=
\bigotimes_{i=1}^{k}
(\bm{D}(i)\oplus\bm{T}(i)).
\end{equation*}

To simplify further formulas, we introduce the notation:
\begin{equation*}
\bm{D}(l,m)
=
\bigotimes_{i=l}^{m}
\bm{D}(i),
\quad
\bm{D}_{j}(l,m)
=
\bigotimes_{i=l}^{m}
\bm{D}_{j}(i),
\quad
\bm{D}_{k}
=
\bm{D}(1,k),
\quad
\bm{D}_{jk}
=
\bm{D}_{j}(1,k),
\qquad
j=1,2;
\end{equation*}
where the empty products are thought of as equal to the identity matrix $\bm{I}$.

The next statement offers lower and upper bounds on the norm $\|\bm{A}_{k}\|$.
\begin{proposition}
Let $\bm{A}(i)$ for all $i=1,\ldots,k$ be matrices defined as \eqref{E-AieqDiTi}. Then, the following double-inequality holds:
\begin{multline}
\|\bm{D}_{1k}\|
\oplus
\|\bm{D}_{2k}\|
\leq
\|\bm{A}_{k}\|
\leq
\|\bm{D}_{1k}\|
\oplus
\|\bm{D}_{2k}\|
\\
\oplus
\bigoplus_{j=1}^{k}
\|\bm{T}(j)\|
\bigoplus_{i=1}^{k}
\|\bm{D}_{1}(1,i-1))\|\|\bm{D}_{2}(i+1,k)\|.
\label{I-normAk}
\end{multline}
\end{proposition}
\begin{proof}
To obtain a lower bound, we use the inequality $\bm{A}(i)=\bm{D}(i)\oplus\bm{T}(i)\geq\bm{D}(i)$, which holds for all $i=1,\ldots,k$. By combining these inequalities, we have 
\begin{equation*}
\|\bm{A}_{k}\|
=
\|\bm{A}(1)\cdots\bm{A}(k)\|
\geq
\|\bm{D}(1)\cdots\bm{D}(k)\|
=
\|\bm{D}(1,k)\|
=
\|\bm{D}_{k}\|.
\end{equation*}

Furthermore, we consider the upper bound. The application of the distributivity of multiplication over addition yields
\begin{equation*}
\bm{A}_{k}
=
\bigotimes_{i=1}^{k}
(\bm{D}(i)\oplus\bm{T}(i))
=
\bm{D}_{k}
\oplus
\bigoplus_{i=1}^{k}
\bm{D}(1,i-1)\bm{T}(i)\bm{D}(i+1,k).
\end{equation*}

We consider the product under summation and apply the properties of the norm to write
\begin{equation*}
\|\bm{D}(1,i-1)\bm{T}(i)\bm{D}(i+1,k)\|
\leq
\|\bm{D}_{1}(1,i-1)\|\|\bm{T}(i)\|\|\bm{D}_{2}(i+1,k)\|.
\end{equation*}

Since $\|\bm{T}(i)\|\leq\|\bm{T}(1)\|\oplus\cdots\oplus\|\bm{T}(k)\|$ for all $i$, we obtain the upper bound:
\begin{equation*}
\|\bm{A}_{k}\|
\leq
\|\bm{D}_{k}\|
\oplus
\bigoplus_{j=1}^{k}
\|\bm{T}(j)\|
\bigoplus_{i=1}^{k}
\|\bm{D}_{1}(1,i-1)\|\|\bm{D}_{2}(i+1,k)\|.
\end{equation*}
 
We note that the matrix $\bm{D}_{k}$ is block diagonal, and hence, $\|\bm{D}_{k}\|=\|\bm{D}_{1k}\|\oplus\|\bm{D}_{2k}\|$. It remains to combine both the lower and upper bounds, which yields \eqref{I-normAk}.
\end{proof}

\section{Stochastic Dynamic Systems}
\label{S-SDS}

In this section, we examine stochastic dynamic systems in the tropical algebra setting, which are used for the description of the evolution of the queueing system under study in the next section. The main purpose of this section is to evaluate the Lyapunov exponent for a dynamic system with a state transition matrix of special form. For further details on the application of tropical algebra to stochastic dynamic systems, one can consult \cite{Heidergott2006Maxplus,Heidergott2006Maxplus-linear}.

We consider a dynamic model that is governed by the state equation represented for all $k=1,2,\ldots$ in terms of max-plus algebra in the form:
\begin{equation} 
\bm{x}(k)
=
\bm{A}^{T}(k)
\bm{x}(k-1),
\qquad
\bm{x}(0)
=
\bm{1},
\label{E-SDE}
\end{equation}
where $\bm{x}(k)$ denotes a state $n$-vector and $\bm{A}(k)$ a state transition $(n\times n)$-matrix given by
\begin{equation*} 
\bm{x}(k)
=
\begin{pmatrix}
x_{1}(k)
\\
\vdots
\\
x_{n}(k)
\end{pmatrix},
\qquad
\bm{A}(k)
=
\begin{pmatrix}
a_{11}(k) & \ldots & a_{1n}(k)
\\
\vdots & \ddots & \vdots
\\
a_{n1}(k) & \ldots & a_{nn}(k)
\end{pmatrix}.
\end{equation*}

Each entry $a_{ij}(k)$ of the matrix $\bm{A}(k)$ may be an RV or a constant. The corresponding random entries in the matrices $\bm{A}(k)$ for $k=1,2,\ldots$ are assumed independent and identically distributed (i.i.d.) with finite expectation and variance. Note that the random entries in one matrix $\bm{A}(k)$ need not be independent.
   
We define the matrix product:
\begin{equation*}
\bm{A}_{k}
=
\bm{A}(1)\cdots\bm{A}(k).
\end{equation*} 

With this notation, the state dynamic equation at \eqref{E-SDE} can be reduced to
\begin{equation*}
\bm{x}(k)
=
\bm{A}_{k}^{T}\bm{x}(0).
\end{equation*} 
	
The Lyapunov exponent indicates the mean growth rate of the state vector, and it is defined as the limit: 
\begin{equation*}
\lambda
=
\lim_{k\to\infty}
\|\bm{x}(k)\|^{1/k}.
\end{equation*} 

We note that, in the context of max-plus algebra, the last definition is represented in the conventional form:
\begin{equation*}
\lambda
=
\lim_{k\to\infty}
\frac{1}{k}
\max(x_{1}(k),\ldots x_{n}(k)).
\end{equation*} 

Furthermore, with $\bm{x}(0)=\bm{1}$ where $\bm{1}=(0,\ldots,0)^{T}$, we have
\begin{equation*}
\|\bm{x}(k)\|
=
\|\bm{A}_{k}^{T}\bm{x}(0)\|
=
\|\bm{A}_{k}^{T}\|
=
\|\bm{A}_{k}\|,
\end{equation*} 
and then, rewrite the above limit as
\begin{equation*}
\lambda
=
\lim_{k\to\infty}
\|\bm{A}_{k}\|^{1/k}.
\end{equation*}

The next result from \cite{Krivulin2009Evaluation,Krivulin2009Methods} (see also \cite{Heidergott2006Maxplus,Heidergott2006Maxplus-linear}), which is a consequence of Kingman's subadditive ergodic theorem \cite{Kingman73Subadditive}, gives conditions for the above limit to exist with probability one (w.~p.~1) and evaluates $\lambda$ as the limit of the expected values of $\|\bm{A}_{k}\|^{1/k}$. 
\begin{theorem}
\label{T-SET}
Let $\{\bm{A}(k)|\ k=1,2,\ldots\}$ be a stationary sequence of random matrices such that $\mathsf{E}\|\bm{A}_{1}\|<\infty$ and $\rho(\mathsf{E}[\bm{A}_{1}])>-\infty$. Then, there exists a finite number $\lambda$ such that
\begin{equation*}
\lim_{k\to\infty}\|\bm{A}_{k}\|^{1/k}
=
\lambda
\quad
\text{w.~p.~1},
\qquad
\lim_{k\to\infty}\mathsf{E}\|\bm{A}_{k}\|^{1/k}
=
\lambda.
\end{equation*} 
\end{theorem}

Since the matrices $\bm{A}(k)$ are assumed to be i.i.d., the sequence of these matrices is stationary. Moreover, since the random entries in $\bm{A}(k)$ have finite expected values, the condition $\mathsf{E}\|\bm{A}_{1}\|<\infty$ holds. The condition $\rho(\mathsf{E}[\bm{A}_{1}])>-\infty$ actually means that the sequence of matrices $\bm{A}_{k}$ does not degenerate into a zero matrix $\bm{0}$ (which has all entries equal to $-\infty$ in max-plus algebra), and it is assumed satisfied.

It follows from Theorem~\ref{T-SET} that, for the dynamic systems under consideration, the Lyapunov exponent exists and can be found as the limit of the expected values $\mathsf{E}\|\bm{A}_{k}\|^{1/k}$ as $k$ tends to $\infty$. The evaluation of the limit and of the expectations $\mathsf{E}\|\bm{A}_{k}\|$ themselves can be a difficult problem. However, it is not difficult to solve the problem for matrices $\bm{A}(k)$ that have a particular form or structure \cite{Krivulin2009Evaluation,Krivulin2009Methods}. Specifically, if the matrices $\bm{A}(k)$ are triangular, then the Lyapunov exponent is calculated as
\begin{equation*}
\lambda
=
\mathop\mathrm{tr}\mathsf{E}[\bm{A}_{1}]
=
\bigoplus_{i=1}^{n}\mathsf{E}[a_{ii}(1)].
\end{equation*} 

In the context of max-plus algebra, the above formula turns into the maximum of the expected values of the diagonal entries in $\bm{A}_{1}=\bm{A}(1)$ given by
\begin{equation*}
\lambda
=
\max_{1\leq i\leq n}\mathsf{E}[a_{ii}(1)].
\end{equation*} 

We note that the same result is valid for the diagonal matrices $\bm{A}(k)$ as well. Moreover, this result can be readily extended to the system \eqref{E-SDE} with block diagonal matrices.
\begin{lemma}
Let $\bm{A}(k)$ for $k=1,2,\ldots$ be block diagonal matrices of the form:
\begin{equation*}
\bm{A}(k)
=
\begin{pmatrix}
\bm{D}_{1}(k) & & \bm{0}
\\
& \ddots & 
\\
\bm{0} & & \bm{D}_{s}(k)
\end{pmatrix}.
\end{equation*}

Consider matrices $\bm{D}_{rk}=\bm{D}_{r}(1)\cdots\bm{D}_{r}(k)$, and suppose that $\|\bm{D}_{rk}\|^{1/k}\rightarrow\mu_{r}$ w.~p.~1 as $k\to\infty$ for all $r=1,\ldots,s$. Then, the Lyapunov exponent of the system \eqref{E-SDE} is given by
\begin{equation*}
\lambda
=
\bigoplus_{r=1}^{s}\mu_{r}.
\end{equation*}
\end{lemma}
\begin{proof}
Since the matrix product $\bm{A}_{k}=\bm{A}(1)\cdots\bm{A}(k)$ has the same block diagonal form as $\bm{A}(k)$, we can write $\|\bm{A}_{k}\|=\|\bm{D}_{1k}\|\oplus\cdots\oplus\|\bm{D}_{sk}\|$. Furthermore, we apply the binomial identity to write the equality $\|\bm{A}_{k}\|^{1/k}=\|\bm{D}_{1k}\|^{1/k}\oplus\cdots\oplus\|\bm{D}_{sk}\|^{1/k}$.  It remains to let $k$ go to $\infty$ on both sides of the equality, which yields the desired result.
\end{proof}

The extension of this result to block triangular matrices seems to be not so easy. Below, we evaluate the Lyapunov exponent for block triangular matrices of special form.  

Consider a dynamic system with state transition matrices $\bm{A}(k)$ of block triangular form defined as \eqref{E-AieqDiTi} in the framework of max-plus algebra. We suppose that the upper diagonal block reduces to an RV $\alpha_{k}$ and the lower is a constant nonrandom matrix $\bm{D}$ to write
\begin{equation*}
\bm{D}_{1}(k)
=
\begin{pmatrix}
\alpha_{k}
\end{pmatrix},
\qquad
\bm{D}_{2}(k)
=
\bm{D}.
\end{equation*}

We assume that $\alpha_{k}$ for $k=1,2,\ldots$ are i.i.d. RVs that have a finite expected value and variance, and the matrix $\bm{D}$ has no zero entries. The RVs $\|\bm{T}(k)\|$ are also assumed i.i.d. with nonnegative expectation and finite variance.

\begin{lemma}
\label{L-Lyapunov}
Let $\mathsf{E}\alpha_{k}=\mu_{1}\geq0$ be the expected value of $\alpha_{k}$ and $\rho(\bm{D})=\mu_{2}>0$ be the spectral radius of $\bm{D}$. Then, the Lyapunov exponent of the system is given by $\lambda=\mu_{1}\oplus\mu_{2}=\max(\mu_{1},\mu_{2})$.
\end{lemma} 
\begin{proof}
To verify the statement, we show that $\mathsf{E}\|\bm{A}_{k}\|^{1/k}\longrightarrow\lambda=\max(\mu_{1},\mu_{2})$ as $k\to\infty$.

We substitute $\bm{D}_{1}(k)=\begin{pmatrix}\alpha_{k}\end{pmatrix}$ and $\bm{D}_{2}(k)=\bm{D}$ into the double-inequality \eqref{I-normAk}, which yields
\begin{equation}
\alpha_{1}\cdots\alpha_{k}
\oplus
\|\bm{D}^{k}\|
\leq
\|\bm{A}_{k}\|
\leq
\alpha_{1}\cdots\alpha_{k}
\oplus
\|\bm{D}^{k}\|
\oplus
\bigoplus_{j=1}^{k}
\|\bm{T}(j)\|
\bigoplus_{i=1}^{k}
(\alpha_{1}\cdots\alpha_{i-1})\|\bm{D}^{k-i}\|.
\label{I-normAk1}
\end{equation}

First, we examine the right inequality. We apply \eqref{I-AklerhoA} to see that $\|\bm{D}^{k}\|\leq\mu_{2}^{k}\|\bm{D}\bm{D}^{-}\|$ and $\|\bm{D}^{k-i}\|\leq\mu_{2}^{k-i}\|\bm{D}\bm{D}^{-}\|\leq\mu_{2}^{k-i+1}\|\bm{D}\bm{D}^{-}\|$. Observing that $\|\bm{D}\bm{D}^{-}\|\geq\|\bm{I}\|=\mathbb{1}$ and $\|\bm{T}(j)\|\geq\mathbb{1}$, we can write the inequalities:
\begin{gather*}
\alpha_{1}\cdots\alpha_{k}
\oplus
\|\bm{D}^{k}\|
\leq
(\alpha_{1}\cdots\alpha_{k}
\oplus
\mu_{2}^{k})
\|\bm{D}\bm{D}^{-}\|
\bigoplus_{j=1}^{k}
\|\bm{T}(j)\|,
\\
\bigoplus_{i=1}^{k}
(\alpha_{1}\cdots\alpha_{i-1})\|\bm{D}^{k-i}\|
\leq
\|\bm{D}\bm{D}^{-}\|
\bigoplus_{i=1}^{k}
(\alpha_{1}\cdots\alpha_{i-1})\mu_{2}^{k-i+1}.
\end{gather*}

With these inequalities, we expand the right inequality at \eqref{I-normAk1} as follows:
\begin{equation*}
\|\bm{A}_{k}\|
\leq
\|\bm{D}\bm{D}^{-}\|
\bigoplus_{j=1}^{k}
\|\bm{T}(j)\|
\bigoplus_{1\leq i+m\leq k}
(\alpha_{1}\cdots\alpha_{i})\mu_{2}^{m}.
\end{equation*}

Next, we rewrite the last inequality in terms of ordinary operations and take expectations. Taking into account that $\mathsf{E}\|\bm{D}\bm{D}^{-}\|=\|\bm{D}\bm{D}^{-}\|$, we obtain
\begin{equation}
\mathsf{E}\|\bm{A}_{k}\|
\leq
\|\bm{D}\bm{D}^{-}\|
+
\mathsf{E}\max_{1\leq j\leq k}
\|\bm{T}(j)\|
+
\mathsf{E}\max_{1\leq i+m\leq k}
(\alpha_{1}+\cdots+\alpha_{i}+m\mu_{2}).
\label{I-normAk2}
\end{equation}

We note that $\|\bm{D}\bm{D}^{-}\|$ is bounded, and hence, $\|\bm{D}\bm{D}^{-}\|/k\rightarrow0$ as $k\to\infty$.

Furthermore, $\|\bm{T}(j)\|$ for $j=1,2,\ldots$ are i.i.d. RVs with finite expectation and variance. As $k$ goes to $\infty$, the expected value of the maximum of these RVs grows as $O(k^{1/2})$ \cite{Gumbel1954Maxima,Hartley1954Universal}, and therefore,
\begin{equation*}
\frac{1}{k}
\mathsf{E}
\max_{1\leq j\leq k}
\|\bm{T}(j)\|
\rightarrow
0.
\end{equation*}

Consider the last term on the right-hand side of \eqref{I-normAk2}, and suppose that $\mathsf{E}\alpha_{1}=\mu_{1}\leq\mu_{2}$. We represent this term as 
\begin{equation*}
\mathsf{E}\max_{1\leq i+m\leq k}
(\alpha_{1}+\cdots+\alpha_{i}+m\mu_{2})
=
k\mu_{2}
+
\mathsf{E}\max_{1\leq i\leq k}
((\alpha_{1}-\mu_{2})+\cdots+(\alpha_{i}-\mu_{2})).
\end{equation*}

We observe that $\alpha_{i}-\mu_{2}$ are i.i.d. RVs with the expectation $\mathsf{E}(\alpha_{i}-\mu_{2})\leq0$ and finite variance. Since the expected value of the maximum of the cumulative sums of these variables grows as $O(k^{1/2})$ as $k$ tends to $\infty$ (see, e.g., \cite{Krivulin2001Evaluation}), we have

\begin{equation*}
\frac{1}{k}
\mathsf{E}\max_{1\leq i+m\leq k}
(\alpha_{1}+\cdots+\alpha_{i}+m\mu_{2})
\rightarrow
\mu_{2}.
\end{equation*}

Using similar arguments, we can verify that if $\mu_{1}\geq\mu_{2}$, then  
\begin{equation*}
\frac{1}{k}
\mathsf{E}\max_{1\leq i+m\leq k}
(\alpha_{1}+\cdots+\alpha_{i}+m\mu_{2})
\rightarrow
\mu_{1}.
\end{equation*}

As a result, we conclude that 
\begin{equation*}
\lambda
=
\lim_{k\to\infty}
\frac{1}{k}
\mathsf{E}\|\bm{A}_{k}\|
\leq
\max(\mu_{1},\mu_{2}).
\end{equation*}

Consider the left inequality at \eqref{I-normAk1}. As $k$ tends to $\infty$, we have $\|\bm{D}^{k}\|^{1/k}\rightarrow\rho(\bm{D})=\mu_{2}$. Moreover, after rewriting the term $(\alpha_{1}\cdots\alpha_{k})^{1/k}$ in terms of the usual operations, we see that
\begin{equation*} 
\frac{1}{k}(\alpha_{1}+\cdots+\alpha_{k})
\rightarrow
\mathsf{E}\alpha_{1}
=
\mu_{1}.
\end{equation*} 

Therefore, the left inequality leads to the inequality:
\begin{equation*}
\lambda
=
\lim_{k\to\infty} 
\frac{1}{k}
\mathsf{E}\|\bm{A}_{k}\|
\geq
\max(\mu_{1},\mu_{2}).
\end{equation*} 

Since the opposite inequality holds, we arrive at the conclusion that
\begin{equation*}
\lambda
=
\max(\mu_{1},\mu_{2}),
\end{equation*} 
which completes the proof.
\end{proof}

It is not difficult to see that this result remains valid if the matrix $\bm{D}$ may have zero entries, but some of its power $\bm{D}^{p}$ is a matrix without zero entries. Indeed, in this case, we can consider a dynamic system:
\begin{equation*}
\bm{x}^{\prime}(k)
=
\bm{A}^{\prime}(k)\bm{x}^{\prime}(k-1),
\end{equation*}
where we use the notation:
\begin{equation*}
\bm{x}^{\prime}(k)
=
\bm{x}(pk),
\qquad
\bm{A}^{\prime}(k)
=
\bm{A}(pk-p+1)\cdots\bm{A}(pk).
\end{equation*}

The matrix $\bm{A}^{\prime}(k)$ has a block triangular form with $\bm{D}^{p}$ as its lower diagonal block.

For this system, we have $\mu_{1}^{\prime}=p\mu_{1}$ and $\mu_{2}^{\prime}=p\mu_{2}$, which yields the Lyapunov exponent $\lambda^{\prime}=\max(p\mu_{1},p\mu_{2})=p\lambda$. Turning back to the initial system, we obtain the solution $\lambda=\max(\mu_{1},\mu_{2})$ provided by the above result.

\section{Application to Battery Swapping and Charging Station Model}
\label{S-ABSCSM}

We are now in a position to apply the previous results to represent the BSCS queueing model in terms of max-plus algebra and evaluate the mean operation cycle time for the model. We start with scalar recurrence Equations \eqref{E-xkeqxk1alphak-ykeqmaxxkyk1cb}, which describe the dynamics of the model, and represent these equations in terms of max-plus algebra. Next, we rewrite the scalar equations in vector form as an implicit vector equation. This equation is solved to obtain an explicit state dynamic equation with a state transition matrix having random entries. Finally, by applying results from the previous sections, we derive an exact formula for calculating the mean cycle time under evaluation. 

\subsection{Tropical Representation of the Model}

Let us rewrite the equations in \eqref{E-xkeqxk1alphak-ykeqmaxxkyk1cb} in terms of max-plus algebra. After replacing the operation $\max$ by the addition $\oplus$ and $+$ by the multiplication $\otimes$ (the sign $\otimes$ is eliminated from the subsequent expressions), the equations become linear in the tropical sense and take the form:
\begin{align*}
x(k)
&=
\alpha_{k}x(k-1),
\\
y(k)
&=
bx(k)\oplus by(k-1)\oplus bcy(k-m)).
\end{align*}

To represent the dynamic system in vector form, we introduce the following vector and matrices (where we use the notation $\mathbb{0}=-\infty$ and $\mathbb{1}=0$):
\begin{gather*}
\bm{v}(k)
=
\begin{pmatrix}
x(k)
\\
y(k)
\\
y(k-1)
\\
\vdots
\\
y(k-m+1)
\end{pmatrix},
\qquad
\bm{B}(k)
=
\begin{pmatrix}
\mathbb{0} & \mathbb{0} & \dots & \dots & \mathbb{0}
\\
b & \mathbb{0} & \dots & \dots & \mathbb{0}
\\
\mathbb{0} & \mathbb{0} & \ddots & & \mathbb{0}
\\
\vdots & \vdots & \ddots & \ddots & \vdots
\\
\mathbb{0} & \mathbb{0} & \dots & \mathbb{0} & \mathbb{0}
\end{pmatrix}
=
\bm{B},
\\
\bm{C}(k)
=
\begin{pmatrix}
\alpha_{k} & \mathbb{0} & \mathbb{0} & \dots & \mathbb{0} & \mathbb{0}
\\
\mathbb{0} & b & \mathbb{0} & \dots & \mathbb{0} & bc
\\
\mathbb{0} & \mathbb{1} & \mathbb{0} & \dots & \mathbb{0} & \mathbb{0}
\\
\vdots & & \ddots & \ddots & & \vdots
\\
\vdots & & & \ddots & \ddots & \vdots
\\
\mathbb{0} & \mathbb{0} & \mathbb{0} & & \mathbb{1} & \mathbb{0}
\end{pmatrix}.
\end{gather*}

With this notation, the system is written as an implicit equation in $\bm{v}(k)$ in the form:
\begin{equation*}
\bm{v}(k)
=
\bm{B}\bm{v}(k)
\oplus
\bm{C}(k)\bm{v}(k-1).
\end{equation*}

We solve this equation for $\bm{v}(k)$ by using Lemma~\ref{L-Axbeqx}. First, we note that $\mathop\mathrm{tr}\bm{B}=\mathbb{0}$. Furthermore, we see that $\bm{B}^{2}=\bm{0}$, and hence, $\bm{B}^{i}=\bm{0}$ for all $i\geq2$. As a result, we have $\mathop\mathrm{Tr}(\bm{B})=\mathbb{0}$ and calculate
\begin{equation*}
\bm{B}^{\ast}
=
\bm{I}
\oplus
\bm{B}
=
\begin{pmatrix}
\mathbb{1} & \mathbb{0} & \dots & \dots & \mathbb{0}
\\
b & \mathbb{1} & & & \mathbb{0}
\\
\mathbb{0} & \mathbb{0} & \ddots & & \mathbb{0}
\\
\vdots & \vdots & & \ddots & 
\\
\mathbb{0} & \mathbb{0} & & & \mathbb{1}
\end{pmatrix}.
\end{equation*}

The application of Lemma~\ref{L-Axbeqx} leads to the explicit state dynamic equation:
\begin{equation*}
\bm{v}(k)
=
\bm{A}^{T}(k)\bm{v}(k-1)
\end{equation*}
with the state transition matrix:
\begin{equation}
\bm{A}^{T}(k)
=
\bm{B}^{\ast}\bm{C}(k)
=
\begin{pmatrix}
\alpha_{k} & \mathbb{0} & \mathbb{0} & \dots & \mathbb{0} & \mathbb{0}
\\
\alpha_{k}b & b & \mathbb{0} & \dots & \mathbb{0} & bc
\\
\mathbb{0} & \mathbb{1} & \mathbb{0} & \dots & \mathbb{0} & \mathbb{0}
\\
\vdots & & \ddots & \ddots & & \vdots
\\
\vdots & & & \ddots & \ddots & \vdots
\\
\mathbb{0} & \mathbb{0} & \mathbb{0} & & \mathbb{1} & \mathbb{0}
\end{pmatrix}.
\label{E-Ak}
\end{equation}

\subsection{Tropical Representation of Performance Measure}

We now exploit the dynamic model derived above to evaluate the mean cycle time $\lambda$ given by \eqref{E-lambda}. First, we see from scalar Equations \eqref{E-xkeqxk1alphak-ykeqmaxxkyk1cb} that the following inequalities are valid:

\begin{equation*}
y(k)
\geq
x(k),
\qquad
y(k)
\geq
y(k-1).
\end{equation*}

As a result, we obtain
\begin{equation*}
y(k)
=
\max(x(k),y(k),\ldots,y(k-m+1)).
\end{equation*}

Since the right-hand side of the above equality coincides with the max-plus algebra norm $\|\bm{v}(k)\|$, we conclude that
\begin{equation*}
y(k)
=
\|\bm{v}(k)\|
=
\|\bm{A}_{k}\|,
\qquad
\bm{A}_{k}
=
\bm{A}(1)\cdots\bm{A}(k).
\end{equation*}

Therefore, the mean cycle time \eqref{E-lambda} can be represented in terms of max-plus algebra as 
\begin{equation*}
\lambda
=
\lim_{k\to\infty}\|\bm{A}_{k}\|^{1/k}.
\end{equation*}

By Theorem~\ref{T-SET}, we can find the mean cycle time by evaluating the limit of the expected values as follows:
\begin{equation*}
\lambda
=
\lim_{k\to\infty}\mathsf{E}\|\bm{A}_{k}\|^{1/k}.
\end{equation*}

\subsection{Evaluation of Mean Cycle Time}

To evaluate the mean cycle time of the system, we apply Lemma~\ref{L-Lyapunov}. Consider the state transition matrix $\bm{A}(k)$ at \eqref{E-Ak}, and note that it has the block triangular form:
\begin{equation*}
\bm{A}(k)
=
\begin{pmatrix}
\alpha_{k} & \alpha_{k}b & \mathbb{0} & \dots & \mathbb{0} & \mathbb{0}
\\
\mathbb{0} & b & \mathbb{1} & & \mathbb{0} & \mathbb{0}
\\
\mathbb{0} & \mathbb{0} & \mathbb{0} & \ddots & \mathbb{0} & \mathbb{0}
\\
\vdots & \vdots & \vdots & \ddots & \ddots & 
\\
\mathbb{0} & \mathbb{0} & \mathbb{0} & \ldots & \mathbb{0} & \mathbb{1}
\\
\mathbb{0} & bc & \mathbb{0} & \ldots & \mathbb{0} & \mathbb{0}
\end{pmatrix}
=
\begin{pmatrix}
\bm{D}_{1}(k) & \bm{T}_{12}(k)
\\
\bm{0} & \bm{D}_{2}(k)
\end{pmatrix},
\end{equation*}
where the matrix blocks are given by
\begin{gather*}
\bm{D}_{1}(k)
=
\begin{pmatrix}
\alpha_{k}
\end{pmatrix},
\qquad
\bm{T}_{12}(k)
=
\begin{pmatrix}
\alpha_{k}b & \mathbb{0} & \ldots & \mathbb{0}
\end{pmatrix},
\\
\bm{0}
=
\begin{pmatrix}
\mathbb{0}
\\
\vdots
\\
\mathbb{0}
\end{pmatrix},
\qquad
\bm{D}_{2}(k)
=
\bm{D}_{2}
=
\begin{pmatrix}
b & \mathbb{1} & & \mathbb{0} & \mathbb{0}
\\
\mathbb{0} & \mathbb{0} & \ddots & \mathbb{0} & \mathbb{0}
\\
\vdots & \vdots & \ddots & \ddots &
\\
\mathbb{0} & \mathbb{0} & \ldots & \mathbb{0} & \mathbb{1}
\\
bc & \mathbb{0} & \ldots & \mathbb{0} & \mathbb{0}
\end{pmatrix}.
\end{gather*}

As is easy to see, the state transition matrix $\bm{A}(k)$ has the same form as in Lemma~\ref{L-Lyapunov} and satisfies the assumptions of this lemma. Moreover, it is not difficult to verify by direct computation that the matrix $\bm{D}^{m-1}$ has no zero entries.

It follows from Lemma~\ref{L-Lyapunov} that the Lyapunov exponent (the mean cycle time) of the system is given by\vspace{-4pt}
\begin{equation*}
\lambda
=
\mu_{1}\oplus\mu_{2},
\qquad
\mu_{1}
=
\mathsf{E}\alpha_{1}
=
a,
\qquad
\mu_{2}
=
\rho(\bm{D}).
\end{equation*}

The evaluation of the spectral radius $\rho(\bm{D})$ by using \eqref{E-rhoA} yields
\begin{equation*}
\mu_{2}
=
b\oplus(bc)^{1/m}.
\end{equation*}

As a result, the mean cycle time is represented in terms of max-plus algebra as
\begin{equation*}
\lambda
=
a
\oplus
b\oplus(bc)^{1/m}.
\end{equation*}

After rewriting in terms of the conventional algebra, we have
\begin{equation}
\lambda
=
\max(a,b,(b+c)/m).
\label{E-lambda1}
\end{equation}

Finally, we note that the obtained formula takes into account the expected value $a=\mathsf{E}\alpha_{1}$ of the random interarrival time of incoming EVs and does not require a complete description of the underlying probability distribution. To illustrate this result by numerical experiments, we used Monte Carlo simulation to estimate the mean cycle time of the system over a long time horizon.

The simulation model consists of the recurrence relations in \eqref{E-xkeqxk1alphak-ykeqmaxxkyk1cb}, which formally describe the evolution of the system. The simulation experiment involves the calculation of the state variables $x(k)$ and $y(k)$ for $k=1,\ldots,K$, where $K$ is sufficiently large. The calculation of $x(k)$ includes sampling from a given probability distribution to fix an interarrival time for each $k$. The estimates $\widehat{\lambda}(k)=y(k)/k$ of the mean cycle time are evaluated for successive $k$ to observe the convergence of the estimates to the value of $\lambda$ specified by \eqref{E-lambda1}.  

In Figures~\ref{F-SRED25}--\ref{F-SRUD30}, we demonstrate the results of estimating the mean cycle time $\lambda$ for a BSCS model with the following parameters fixed: the swapping and charging time of one BP were set to $b=5$ and $c=100$ and the number of BPs available at the station to $m=4$. As an interarrival time distribution for EVs, the exponential and uniform distributions with means $a=25$ and $a=30$ were considered. The number of EVs was set to $K=200$. 

Figures~\ref{F-SRED25} and \ref{F-SRED30} show the results of evaluating the estimators $\widehat{\lambda}(k)=y(k)/k$ for $k$ from $1$ to $200$ step $5$ 
by the simulation of the system with the exponential interarrival time with mean $a=25$ and $a=30$, respectively. The simulation results for the same BSCS with the interarrival time uniformly distributed over $[5,45]$ (with mean $a=25$) and $[10,50]$ (with mean $a=30$) are given in Figures~\ref{F-SRUD25} and \ref{F-SRUD30}.

\begin{figure}[ht]
\begin{center}
\begin{tikzpicture}

\begin{axis}[legend pos=south east,
width=14cm,
height=9cm,
grid=major,
ymin=20,
ymax=50,
xmin=1,
xmax=200
]

\addplot[
mark=*
]
coordinates {
(5, 42.0000)
(10, 32.0000)
(15, 28.6667)
(20, 27.0000)
(25, 29.4000)
(30, 28.1667)
(35, 27.2857)
(40, 26.6250)
(45, 28.0000)
(50, 27.4000)
(55, 26.9091)
(60, 26.5000)
(65, 27.4615)
(70, 27.0714)
(75, 26.7333)
(80, 26.4375)
(85, 27.1765)
(90, 26.8889)
(95, 26.6316)
(100, 26.4000)
(105, 27.0000)
(110, 26.7727)
(115, 26.5652)
(120, 26.3750)
(125, 26.8800)
(130, 26.6923)
(135, 26.5185)
(140, 26.3571)
(145, 26.7931)
(150, 26.6333)
(155, 26.4839)
(160, 26.3438)
(165, 26.7273)
(170, 26.5882)
(175, 26.4571)
(180, 26.3333)
(185, 26.6757)
(190, 26.5526)
(195, 26.4359)
(200, 26.3250)
};

\addlegendentry{$\widehat{\lambda}(k)$}

\addplot[
blue,
line width=1.75pt,
domain=0:200,
y domain=20:50
]
coordinates {
(0,26.25)
(200,26.25)
};

\addlegendentry{$\lambda=26.25$}

%
%
%
%
\node[style={fill=white}] at (axis cs: 80,45) {$a=25$, $b=5$, $c=100$, $m=4$};
\node[style={fill=white}] at (axis cs: 80,40) {$\lambda=\max(a,b,(b+c)/m)=26.25$};

\end{axis}

\end{tikzpicture}
\caption{Simulation results for the exponential distribution with expected value of $25$.} 
\label{F-SRED25}
\end{center}
\end{figure}
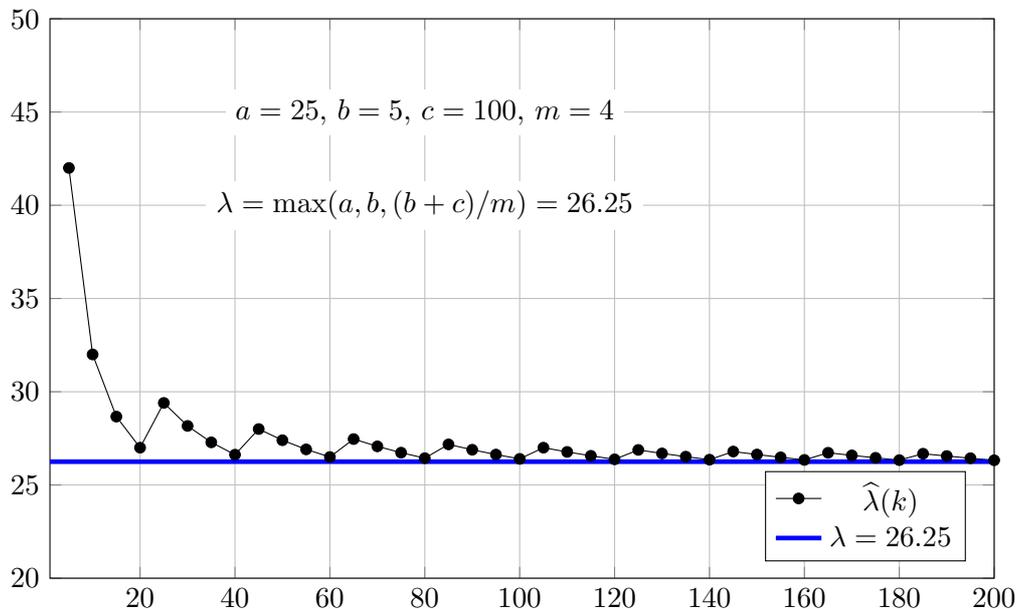

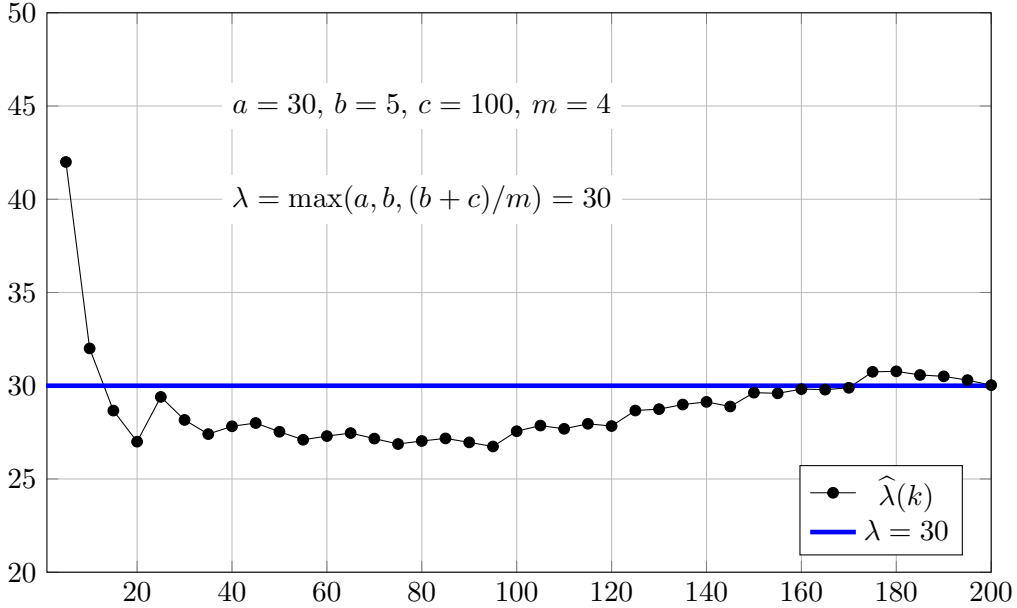
\begin{figure}[ht]
\begin{center}
\begin{tikzpicture}

\begin{axis}[legend pos=south east,
width=14cm,
height=9cm,
grid=major,
ymin=20,
ymax=50,
xmin=1,
xmax=200
]

\addplot[
mark=*
]
coordinates {
(5, 42.0000)
(10, 32.0000)
(15, 28.6667)
(20, 27.0000)
(25, 29.4000)
(30, 28.1667)
(35, 27.4067)
(40, 27.8256)
(45, 28.0000)
(50, 27.5329)
(55, 27.1018)
(60, 27.3004)
(65, 27.4615)
(70, 27.1663)
(75, 26.8747)
(80, 27.0378)
(85, 27.1765)
(90, 26.9627)
(95, 26.7432)
(100, 27.5615)
(105, 27.8639)
(110, 27.6903)
(115, 27.9591)
(120, 27.8387)
(125, 28.6710)
(130, 28.7452)
(135, 28.9914)
(140, 29.1334)
(145, 28.8874)
(150, 29.6298)
(155, 29.5928)
(160, 29.8173)
(165, 29.7910)
(170, 29.8898)
(175, 30.7468)
(180, 30.7724)
(185, 30.5783)
(190, 30.5035)
(195, 30.3015)
(200, 30.0320)
};

\addlegendentry{$\widehat{\lambda}(k)$}

\addplot[
blue,
line width=1.75pt,
domain=0:200,
y domain=20:50
]
coordinates {
(0,30)
(200,30)
};

\addlegendentry{$\lambda=30$}

%
%
%
%
\node[style={fill=white}] at (axis cs: 80,45) {$a=30$, $b=5$, $c=100$, $m=4$};
\node[style={fill=white}] at (axis cs: 80,40) {$\lambda=\max(a,b,(b+c)/m)=30$};

\end{axis}

\end{tikzpicture}
\caption{Simulation results for the exponential distribution with expected value of $30$.}
\label{F-SRED30}
\end{center}
\end{figure}

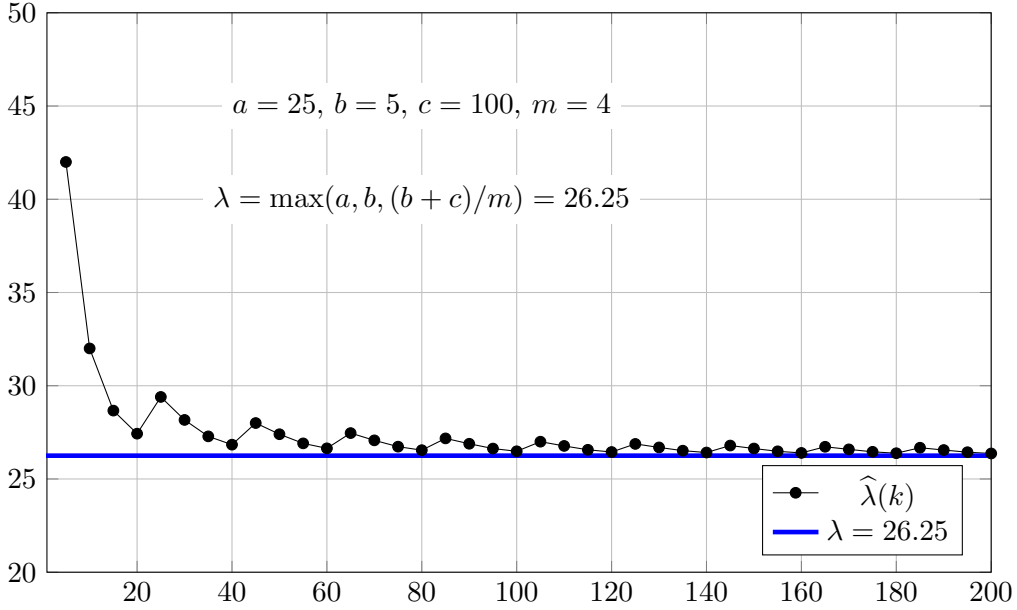
\begin{figure}[ht]
\begin{center}
\begin{tikzpicture}

\begin{axis}[legend pos=south east,
width=14cm,
height=9cm,
grid=major,
ymin=20,
ymax=50,
xmin=1,
xmax=200
]

\addplot[
mark=*
]
coordinates {
(5, 42.0000)
(10, 32.0000)
(15, 28.6667)
(20, 27.4297)
(25, 29.4000)
(30, 28.1667)
(35, 27.2857)
(40, 26.8398)
(45, 28.0000)
(50, 27.4000)
(55, 26.9091)
(60, 26.6432)
(65, 27.4615)
(70, 27.0714)
(75, 26.7333)
(80, 26.5449)
(85, 27.1765)
(90, 26.8889)
(95, 26.6316)
(100, 26.4859)
(105, 27.0000)
(110, 26.7727)
(115, 26.5652)
(120, 26.4466)
(125, 26.8800)
(130, 26.6923)
(135, 26.5185)
(140, 26.4185)
(145, 26.7931)
(150, 26.6333)
(155, 26.4839)
(160, 26.3975)
(165, 26.7273)
(170, 26.5882)
(175, 26.4571)
(180, 26.3811)
(185, 26.6757)
(190, 26.5526)
(195, 26.4359)
(200, 26.3680)
};

\addlegendentry{$\widehat{\lambda}(k)$}

\addplot[
blue,
line width=1.75pt,
domain=0:200,
y domain=20:50
]
coordinates {
(0,26.25)
(200,26.25)
};

\addlegendentry{$\lambda=26.25$}

%
%
%
%
\node[style={fill=white}] at (axis cs: 80,45) {$a=25$, $b=5$, $c=100$, $m=4$};
\node[style={fill=white}] at (axis cs: 80,40) {$\lambda=\max(a,b,(b+c)/m)=26.25$};

\end{axis}

\end{tikzpicture}
\caption{{Simulation results} 
 for the uniform distribution over $[5,45]$.}
\label{F-SRUD25}
\end{center}
\end{figure}

We observed that the simulation results presented demonstrate little difference between the estimates for both the exponential and uniform distributions with the same mean $a=25$. It can be explained by the high arrival rate of EVs, which makes, in this case, the BP recharging and swapping process a dominated time factor that diminishes the influence of random fluctuations in the interarrival time.

\begin{figure}[ht]
\begin{center}
\begin{tikzpicture}

\begin{axis}[legend pos=south east,
width=14cm,
height=9cm,
grid=major,
ymin=20,
ymax=50,
xmin=1,
xmax=200
]

\addplot[
mark=*
]
coordinates {
(5, 42.0000)
(10, 35.3017)
(15, 33.7061)
(20, 31.8570)
(25, 30.9129)
(30, 29.9412)
(35, 29.8494)
(40, 29.8937)
(45, 29.7359)
(50, 29.2479)
(55, 28.9830)
(60, 28.9241)
(65, 28.6633)
(70, 28.5405)
(75, 28.9555)
(80, 29.1730)
(85, 29.6788)
(90, 29.7295)
(95, 29.5454)
(100, 29.2346)
(105, 29.2286)
(110, 29.2851)
(115, 29.5788)
(120, 30.2022)
(125, 30.0836)
(130, 30.1020)
(135, 30.2152)
(140, 30.1501)
(145, 30.0677)
(150, 30.1574)
(155, 30.2697)
(160, 30.1730)
(165, 30.0028)
(170, 30.0331)
(175, 29.9851)
(180, 29.9798)
(185, 29.9593)
(190, 30.0432)
(195, 29.9563)
(200, 29.9339)
};

\addlegendentry{$\widehat{\lambda}(k)$}

\addplot[
blue,
line width=1.75pt,
domain=0:200,
y domain=20:50
]
coordinates {
(0,30)
(200,30)
};

\addlegendentry{$\lambda=30$}

%
%
%
%
\node[style={fill=white}] at (axis cs: 80,45) {$a=30$, $b=5$, $c=100$, $m=4$};
\node[style={fill=white}] at (axis cs: 80,40) {$\lambda=\max(a,b,(b+c)/m)=30$};

\end{axis}

\end{tikzpicture}
\caption{Simulation results for the uniform distribution over $[10,50]$.}
\label{F-SRUD30}
\end{center}
\end{figure}
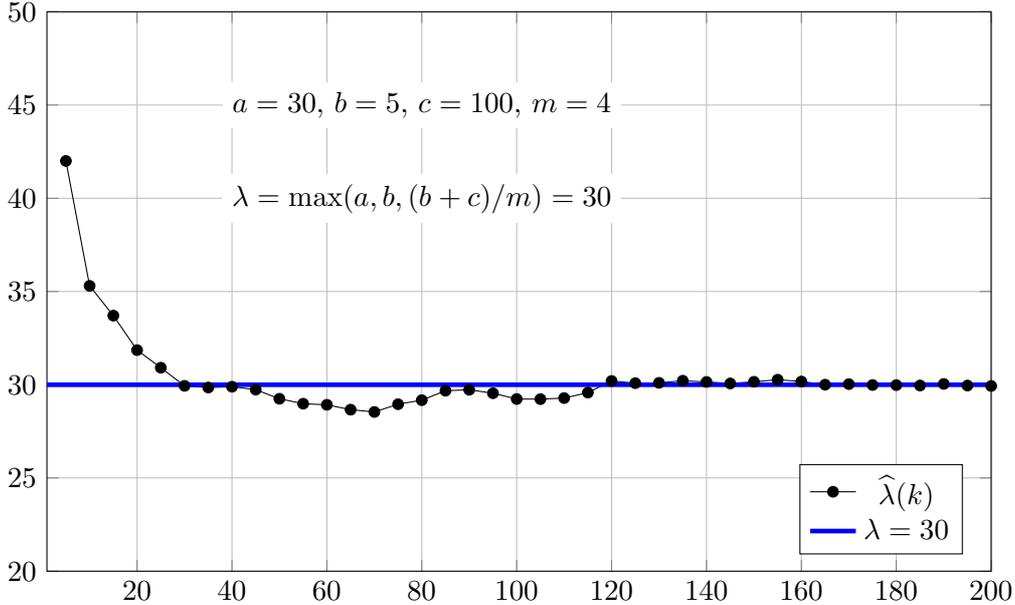

\section{Example of Application Problem}
\label{S-EAP}

In this section, we offer an example of the application of the obtained results to solve real-world problems. Consider a network that consists of $N$ BSCSs. For each station $i=1,\ldots,N$, let $a_{i}$ be the mean interarrival time of EVs. We denote the swapping time and the charging time of one BP by $b_{i}$ and $c_{i}$, respectively.

Assume that the BSCS is equipped with $m_{i}$ BPs intended for swapping, and examine the mean cycle time for station $i$, which is given by 
\begin{equation*}
\lambda_{i}
=
\max(a_{i},b_{i},(b_{i}+c_{i})/m_{i}).
\end{equation*}

The mean swapping rate at the station is evaluated as $1/\lambda_{i}$, whereas the mean number of BPs swapped for a large time horizon $T$ is $T/\lambda_{i}$. 

Let us suppose that one swapping at station $i$ generates an income $r_{i}$. Then, the mean total income during time $T$ is equal to
\begin{equation*}
\frac{r_{i}T}{\max(a_{i},b_{i},(b_{i}+c_{i})/m_{i})}.
\end{equation*}

We represent the mean total income as a function of the number $m$ of BPs at the station in the form
\begin{equation*}
R_{i}(m)
=
\frac{r_{i}T}{\max(a_{i},b_{i},(b_{i}+c_{i})/m)}
=
\begin{cases}
\displaystyle{\frac{r_{i}T}{b_{i}+c_{i}}m}, & \text{if $\displaystyle{0\leq m\leq \frac{b_{i}+c_{i}}{\max(a_{i},b_{i})}}$};
\\
\displaystyle{\frac{r_{i}T}{\max(a_{i},b_{i})}}, & \text{if $\displaystyle{m>\frac{b_{i}+c_{i}}{\max(a_{i},b_{i})}}$}.
\end{cases}
\end{equation*}

It follows from the representation that the function $R_{i}(m)$ increases until $m$ becomes greater than a threshold value $(b_{i}+c_{i})/\max(a_{i},b_{i})$, and remains unchanged with further increase of $m$. 
The maximum mean total income and corresponding optimal number of BPs are defined as
\begin{equation*}
R_{i}(m^{\ast})
=
\frac{r_{i}T}{b_{i}+c_{i}}
\left[
\frac{b_{i}+c_{i}}{\max(a_{i},b_{i})}
\right],
\qquad
m^{\ast}
=
\left[
\frac{b_{i}+c_{i}}{\max(a_{i},b_{i})}
\right],
\end{equation*}
where $[x]$ denotes the integer part of $x$.

Suppose there are $M$ BPs, which we need to distribute between the BSCSs in the network so as to minimize (maximize) an appropriate optimality criterion. If the purpose is to maximize the mean total income generated by the network, the problem is formulated to find the number $m_{i}$ of BPs for each station $i$ to attain the maximum:
\begin{equation*}
\begin{aligned}
\max_{m_{1},\ldots,m_{N}>0}
&&&
\sum_{i=1}^{N}\frac{r_{i}}{\max(a_{i},b_{i},(b_{i}+c_{i})/m_{i})};
\\
\text{s.t.}
&&&
m_{1}+\cdots+m_{N}
=
M.
\end{aligned}
\end{equation*}
 
As a reasonable approximate solution technique, we can define the optimal numbers $m_{i}$ to be proportional to $w_{i}=r_{i}/(b_{i}+c_{i})$. With this technique, the number $m_{i}$ is first found for each $i=1,\ldots,N$ as the nearest positive integer:
\begin{equation*}
m_{i}
\approx
w_{i}/(w_{1}+\cdots+w_{N}).
\end{equation*}

Furthermore, we check whether the numbers $m_{i}$ are outside their threshold values or not. If, for each $i$, the inequality $m_{i}\leq(b_{i}+c_{i})/\max(a_{i},b_{i})$ holds, then the obtained numbers $m_{i}$ are taken as a solution to the problem. 

Suppose that $m_{i}>(b_{i}+c_{i})/\max(a_{i},b_{i})$ for some $i$. In this case, we decrement $m_{i}$ by one and increment some $m_{j}$ such that
\begin{equation*}
j
=
\arg\max_{k\ne i}\{w_{k}|\ m_{k}<(b_{k}+c_{k})/\max(a_{k},b_{k})-1\}.
\end{equation*}

We continue to redistribute BPs between stations until all stations have the number of BPs within their threshold values.



\section{Conclusions}
\label{S-S}

In this paper, we propose a new approach to the analysis of BSCSs' operation, which combines queueing modeling with the application of the methods and results of tropical algebra. We started with the development of a queueing model in the form of a system of recurrence equations that determine the dynamics of a BSCS. We introduced a related performance measure in the form of the mean operation cycle time. Then, the model was represented in terms of max-plus algebra as a linear vector dynamic system with a random state transition matrix, whereas the performance measure became the Lyapunov exponent of the system. We applied the methods and techniques of tropical algebra together with the results on the convergence of the expected value of the maximums of random variables to find the Lyapunov exponent as a limit of the expected value of the matrix norms. After the calculation of the Lyapunov exponent, we arrived at an explicit expression in terms of the expected values of the random variables and constants involved. We showed how this expression can be used to evaluate and optimize the performance of the BSCSs' operation.

We believe that the described research demonstrates the strong potential of the proposed approach to investigate various dynamic models that can be represented as stochastic linear dynamic systems in the tropical algebra setting. The results obtained indicated the ability of the approach to supplement and complement existing techniques of the modeling and optimization of BSCSs' operation. The approach offers a potential to provide explicit results that are given in terms of the expected values of the the RVs involved in the dynamic model and do not require a specific probability distribution. At the same time, the approach can be applied only to dynamic models that are linear in the tropical algebra sense such as queueing models, whose dynamics can be described in terms of the operations of the maximum and addition. This constitutes one of the limitations of the obtained solution, which may make it difficult to extend this approach to other classes of queueing models. Another limitation is that the approach focuses on evaluating the Lyapunov exponents of dynamic systems and can be hardly extended to other performance measures of interest. 

Possible directions of future research may concern further investigation of the obtained solution, including the sensitivity analysis of the model. An extension of the BSCS model to incorporate more-complicated operation patterns and accommodate additional constraints are of particular interest. As an example, one can consider a station where the number of simultaneously charged BPs is limited or the battery charging time is random. The formulation of new meaningful optimization problems to improve BSCSs' performance and the development of efficient solutions constitute another promising line of investigation.

\section*{Acknowledgments}
The authors are very grateful to the anonymous reviewers for their valuable comments and suggestions, which have been incorporated into the revised manuscript.

\bibliographystyle{abbrvurl}

\bibliography{Tropical_modelling_of_battery_swapping_and_charging_station}

\end{document}